\documentclass{amsart}
\usepackage[dvipdfmx]{graphicx}
\usepackage{color}
\usepackage{overpic}
\usepackage{enumerate}

\newtheorem{theorem}{Theorem}[section]
\newtheorem{proposition}[theorem]{Proposition}
\newtheorem{lemma}[theorem]{Lemma}

\theoremstyle{definition}

\theoremstyle{remark}
\newtheorem{remark}[theorem]{Remark}

\newcommand{\sgn}{\mathop{\mathrm{sgn}}\nolimits}

\author{Kazuhiro Ichihara}
\address{Department of Mathematics, 
College of Humanities and Sciences, Nihon University, 
3-25-40 Sakurajosui, Setagaya-ku, Tokyo 156-8550, Japan.}
\email{ichihara.kazuhiro@nihon-u.ac.jp}

\author{In Dae Jong}
\address{Department of Mathematics, Kindai University, 3-4-1 Kowakae, Higashiosaka City, Osaka 577-0818, Japan} 
\email{jong@math.kindai.ac.jp}

\author{Hidetoshi Masai}
\address{Department of Mathematics Tokyo Institute of Technology 
2-12-1 Ookayama, Meguroku, Tokyo 152-8551, Japan.} 
\email{masai@math.titech.ac.jp}

\dedicatory{Dedicated to Professor Masaaki Wada on the occasion of his 60th birthday}

\title{Complete exceptional surgeries on two-bridge links}



\date{\today}

\begin{document}

\begin{abstract}
We give a list of hyperbolic two-bridge links which includes all such links with complete exceptional surgeries, i.e., Dehn surgeries on both components which yield non-hyperbolic manifolds but whose all the proper sub-fillings give hyperbolic manifolds. 
Also all the candidate slopes of complete exceptional surgeries for them are enumerated in our lists. 
\end{abstract}

\maketitle

\section{Introduction}\label{sec:intro}

Given an $n$-component link $L$ in  a 3-manifold $M$, the following operation is called a \textit{Dehn surgery on $L$}: removing an open tubular neighborhood $\text{Int}N(L)$ of $L$ from $M$, and gluing $n$ solid tori to $M- \text{Int}N(L)$ along the torus boundary components $\partial N( L )$. 
The operation is called a \textit{Dehn surgery on a component $L_i$ of $L$} if we glue only one solid torus along $\partial N(L_i)$ after removing $\text{Int}N(L)$. 
One of the motivations to study Dehn surgery is given by the famous result \cite[Theorem 5.8.2]{Thurston1978}: 
On each component of a hyperbolic link,  there are only finitely many Dehn surgeries yielding non-hyperbolic manifolds. 
Here, a 3-manifold is said to be \textit{hyperbolic} if it admits a complete hyperbolic metric of finite volume, and a link is called hyperbolic if its complement is a hyperbolic 3-manifold. 
In view of this, a Dehn surgery on a hyperbolic link giving a non-hyperbolic manifold is said to be an \textit{exceptional surgery}. 

In the research of exceptional surgery, related to knot theory, 
there are several studies of exceptional surgeries for well-known classes of links in the 3-sphere $S^3$. 
In this paper, as an extension of studies for two-bridge links in $S^3$ \cite{GodaHayashiSong2009, Ichihara2012, Wu1999}, 
we give a list of hyperbolic two-bridge links 
which includes all such links admitting complete exceptional surgeries. 
Also given are all the candidates of \textit{surgery slopes}, that is, the slopes determined by the meridians of the attached solid tori. 
The completeness of our list will be argued in the next project. 

Here, we call a Dehn surgery on a two-bridge link $L$ along surgery slopes $(\gamma_1, \gamma_2)$ a \emph{complete exceptional surgery} if it produces a closed non-hyperbolic 3-manifold, but the surgery on one component of $L$ along $\gamma_i$ yields a hyperbolic manifold with a single cusp for $i=1,2$. 
We remark that exceptional surgeries on one component of two-bridge links are already classified in \cite{Ichihara2012}. 

Recall that a link in $S^3$ is called a \textit{two-bridge link} if it admits a diagram with exactly two maxima and two minima. 
Concerning two-bridge links, we follow the notation in \cite{GodaHayashiSong2009, Ichihara2012, Wu1999}. 
 That is, for a continued fraction $[a_1, \dots , a_k]$ with non-zero integers $a_1, \dots, a_k$, let $L_{[a_1, \dots , a_k]}$ denote the two-bridge link in $S^3$ represented by the diagram in Figure~\ref{fig:2bridge}. 
Also see Figure~\ref{fig:twists}. 
If $[a_1, \dots , a_k] = p/q$, then we also denote the link by $L_{p/q}$, and call the two-bridge link of type $p/q$.
Note that, following the references \cite{FloydHatcher, HatcherThurston1985}, we denote by $[a_1, . . . , a_k]$ the following subtractive continued fraction: 
\[ \dfrac{1}{a_1 - \dfrac{1}{a_2 - \dfrac{1}{a_3 - \cdots -\dfrac{1}{a_k}}}} \]

Then our main theorem is the following. 

\begin{theorem}\label{thm:main}
If a hyperbolic two-bridge link $L$ in $S^3$ admits a complete exceptional surgery along the slopes $(\gamma_1, \gamma_2)$, then $L$ with $(\gamma_1 , \gamma_2)$ is equivalent to one of those given in Tables 1--6 in Section~\ref{sec:computer}. 
\end{theorem}

\begin{figure}[!htb]
\centering
\begin{overpic}[width=.5\textwidth]{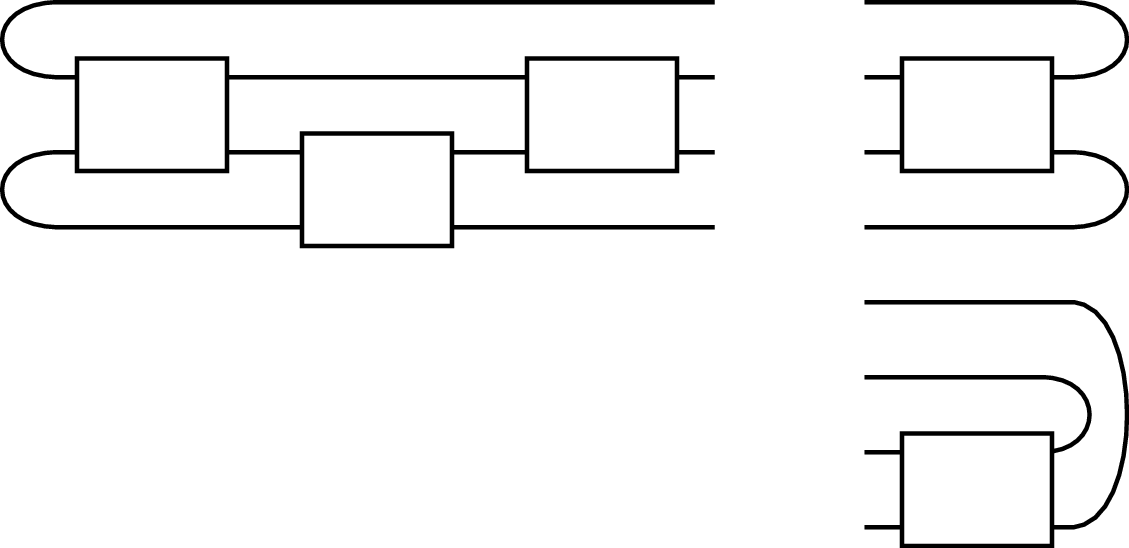}
\put(11.5,37){$a_1$} 
\put(51.5,37){$a_3$} 
\put(67,37){$\cdots$} 
\put(67,11){$\cdots$} 
\put(84.5,37){$a_k$} 
\put(84.5,4){$a_k$} 
\put(31.5,30.5){$a_2$} 
\put(103,37){($k$ is odd)} 
\put(103,10){($k$ is even)} 
\end{overpic}
\caption{A diagram of a two-bridge link $L_{[a_1, \dots, a_k]}$. } 
\label{fig:2bridge}
\end{figure}

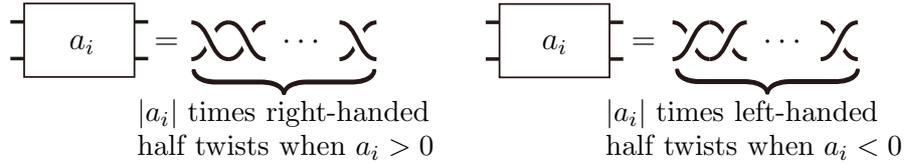
\begin{figure}[!htb]
\centering
\begin{overpic}[width=.7\textwidth]{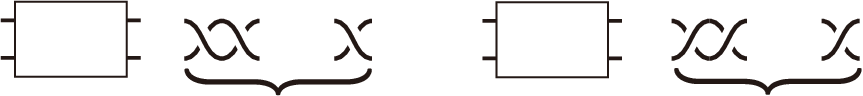}
\put(7,5.5){$a_i$} 
\put(17,5.5){$=$} 
\put(32,5.5){$\cdots$} 
\put(62.5,5.5){$a_i$} 
\put(73,5.5){$=$} 
\put(88.5,5.5){$\cdots$} 
\put(15,-3){{\small $|a_i|$ times right-handed}}
\put(15,-7){{\small half twists when $a_i>0$}}
\put(70,-3){{\small $|a_i|$ times left-handed}}
\put(70,-7){{\small half twists when $a_i<0$}}
\end{overpic}\bigskip\bigskip
\caption{The right-handed twists when $a_i >0$, and the left-handed twists when $a_i < 0$.} 
\label{fig:twists}
\end{figure}

Our study in this paper is motivated by the results given in \cite{GodaHayashiSong2009, Ichihara2012, Wu1999}. 
In fact, our approach to obtain Theorem~\ref{thm:main} is based on the same technique as in \cite{Wu1999}, that is, using an essential branched surface. 
The other technique we use is the computer-aided search of exceptional surgeries on hyperbolic links developed in \cite{hikmot} and utilized by the first- and the third-named authors in \cite{IchiharaMasai}. 

This paper is organized as follows. 
In Section~\ref{sec:proof}, 
we propose a theorem as a key step toward Theorem~\ref{thm:main}. 
In fact, we give a list of families of hyperbolic two-bridge links containing all of them 
admitting complete exceptional surgeries (Theorem~\ref{thm:channel}). 
In Section~\ref{sec:computer}, by using a computer, we give a list of candidates of complete exceptional surgeries on two-bridge links listed in Theorem~\ref{thm:channel}. 
In Appendix, we will give proofs of two elementary algebraic lemmas used in the proof of Theorem~\ref{thm:channel}.

\section{Constraints from essential branched surfaces}\label{sec:proof} 


In this section, $L_{p/q}$ denotes a hyperbolic two-bridge link of type $p/q$, 
and thus, $p$ is odd and $q$ is non-zero even. 
The purpose of this section is to show the following. 

\begin{theorem}\label{thm:channel} 
If $L_{p/q}$ admits a complete exceptional surgery, 
then $L_{p/q}$ is equivalent to one of the following: 
\begin{enumerate}
\item[{\rm (a-1)}]
$L_{[2m+1, 2n -1]}$ with $m \ge 1$, $n \ne 0, 1$. 

\item[{\rm (b-1)}] 
$L_{[2m, 2n, 2l]}$ with $m \ge 1$,  $|n| \ge 2$, $|l| \ge 2$. 

\item[{\rm (b-2)}] 
$L_{[2m, 2n-1, -2l]}$ with $m \ge 1$,  $|n| \ge 2$, $l \ge 1$. 

\item[{\rm (b-3)}] 
$L_{[2m, 2n+1, 2l]}$ with $m \ge 1$,  $|n| \ge 2$, $l \ge 1$. 

\item[{\rm (b-4)}] 
$L_{[2m+1, 2n, 2l-1]}$ with $m \ge 1$, $n \ne 0$, $l \ne 0,1$. 

\item[{\rm (c-1)}] 
$L_{[2m+1, 2n, -2 \sgn(l), 2l-1]}$ with $m\ge 1$, $n \ne 0$, $l \ne 0, 1$. 

\item[{\rm (c-2)}] 
$L_{[2m+1, 2n-1, -2\sgn(l), 2l]}$ with $m\ge 1$, $n \ne 0,1$, $l \ne 0$. 

\end{enumerate} 
Here $\sgn(l)$ denotes $1$ (resp.\ $-1$) when $l$ is positive (resp.\ negative). 
In addition, in {\rm (b-1)}, {\rm (b-2)} and {\rm (b-3)}, if $m =1$, then $n \le -2$ holds. 
\end{theorem} 



\subsection{Proof of Theorem~\ref{thm:channel}}
We first give a proof of Theorem~\ref{thm:channel} assuming several statements 
which will be proved in later subsections. 

One of the key ingredients to prove Theorem~\ref{thm:channel} is 
Delman's construction~\cite{Delman-unpub} 
of essential branched surfaces 
which are described in terms of ``allowable paths'' with ``channels''. 
In Subsection~\ref{subsec:Delman}, 
we will recall the definitions of them and 
give a brief review to introduce the following lemma, which might be well-known to experts in this area. 

\begin{lemma}\label{lem:3ch} 
If there exists an allowable path 
for $p/q$ 
containing 
three channels, 
then $L_{p/q}$ admits no complete exceptional surgery. 
\end{lemma}




In Subsection~\ref{subsec:Channelidx}, 
we observe how to find channels in an allowable path 
by using ``channel indices''. 
In fact, we have the following which will be proved in Subsection~\ref{subsec:Channelidx}. 


\begin{proposition}\label{clm:3channels} 
If there are three channel indices for an even continued fraction $[b_1,\dots, b_k]$, 
then there exists an allowable path for 
$p/q = [b_1,\dots, b_k]$ 
with three channels except for the cases where 
$[b_1,\dots, b_k] = [b_1, 2, \dots, 2, 4, 2, \dots, 2]$ or 
$[b_1,\dots, b_k] = [b_1, -2, \dots, -2, -4, -2, \dots, -2]$. 
\end{proposition}

Here an \emph{even continued fraction} $[b_1, \dots, b_k]$ is a continued fraction such that all $b_i$'s are even and the length $k$ is odd and grater than two. 
The following two lemmas are 
elementary algebraic and independent of the other arguments. 
Thus, we put their proofs in Appendix.

\begin{lemma}\label{clm:ch3Exception} 
Let $a$ and $a'$ be even integers with $a \ge 4$ and $a' \ge 2$. 
Each of the even continued fractions 
$[a, 2,\dots,2, 4, 2,\dots,2]$ and  
$[a', -2,\dots,-2, -4, -2,\dots,-2]$ 
is expressed by one of the following continued fractions: 
\begin{enumerate}[{\rm (1)}] 
\item 
$[2m+1, 2n, 2, 2l -1]$ with $m \ge 1$,  $n \le -1$, $l \le -1$. 

\item 
$[2m+1, 2n-1, 2, 2l]$ with $m \ge 1$,  $n \le -1$, $l \le -1$. 

\item 
$[2m+1, 2n, -2, 2l -1]$ with $m \ge 1$,  $n \ge 1$, $l \ge 2$. 

\item 
$[2m+1, 2n-1, -2, 2l]$ with $m \ge 1$,  $n \ge 2$, $l \ge 1$. 
\end{enumerate} 
\end{lemma} 

\begin{lemma}\label{clm:ch2} 
Let $[b_1, \dots, b_k]$ be an even continued fraction of $p/q$ 
with at most two channel indices and $k \ge 3$. 
Then $p/q$ can be expressed by one of the following continued fractions: 
\begin{enumerate}[{\rm (1)}] 
\item[{\rm (0)}] 
$[2m+1, 2n-1]$ with $m \ge 1$, $n \ne 0,1$

\item 
$[2m, 2n, 2l]$ with $m \ge 1$,  $|n| \ge 2$, $|l| \ge 2$. 

\item 
$[2m, 2n-1, -2l]$ with $m \ge 1$,  $|n| \ge 2$, $l \ge 1$. 

\item 
$[2m, 2n+1, 2l]$ with $m \ge 1$,  $|n| \ge 2$, $l \ge 1$. 


\item 
$[2m+1, 2n, 2l-1]$ 
with $m \ge 1$, $n \ne 0$, $ l \ne 0,1$. 

\item 
$[2m+1, 2n, -2, 2l-1]$ with $m \ge 1$, $n \le -1$, $l \ge 2$. 

\item 
$[2m+1, 2n-1, -2, 2l]$ with $m \ge 1$, $n \le -1$, $l \ge 1$. 

\item 
$[2m+1, 2n, 2, 2l-1]$ with $m\ge 1$, $n \ge 1$, $l \le -1$. 

\item 
$[2m+1, 2n-1, 2, 2l]$ with $m\ge 1$, $n \ge 2$, $l \le -1$. 
\end{enumerate}
In addition, in {\rm (1)--(3)}, if $m=1$, then $n \le -2$. 
\end{lemma}

Here we give a proof of Theorem~\ref{thm:channel} 
assuming Proposition~\ref{clm:3channels} and 
Lemmas~\ref{lem:3ch}, \ref{clm:ch3Exception}, and \ref{clm:ch2}. 

\begin{proof}[Proof of Theorem~\ref{thm:channel}] 
Let $L_{p/q}$ be a hyperbolic two-bridge link. 
Since $q$ is non-zero even and $p$ is odd, it is known that $p/q$ can be expressed by an even continued fraction $[b_1, \dots, b_k]$, that is, all $b_i$'s are even, of odd length $k$. 
Since $L_{p/q}$ is assumed to be hyperbolic, it is a non-torus two-bridge link, and thus, we may assume that $k \ge 3$. 
Then, by Proposition~\ref{clm:3channels}, 
if there are three channel indices for $[b_1, \dots, b_k]$ and 
if $[b_1, \dots, b_k]$ does not coincide with $[b_1, 2, \dots, 2, 4, 2, \dots, 2]$ 
or $[b_1, -2, \dots, -2, -4, -2, \dots, -2]$, 
then we obtain an allowable path 
for $p/q$ with three channels.  
Then, by Lemma~\ref{lem:3ch}, 
$L_{p/q}$ admits no complete exceptional Dehn surgery. 
Considering the contrapositive, 
if $L_{p/q}$ admits such a Dehn surgery, 
then either 
\begin{itemize} 
\item 
$p/q = [b_1, 2, \dots, 2, 4, 2, \dots, 2]$ or 
\item 
$p/q = [b_1, -2, \dots, -2, -4, -2, \dots, -2]$ or 
\item 
$p/q = [b_1, \dots, b_k]$ has at most two channel indices with $k \ge 3$. 
\end{itemize} 

For the first two even continued fractions, we have Lemma~\ref{clm:ch3Exception}. 
To complete the proof of Theorem~\ref{thm:channel}, 
we have to enumerate even continued fractions with at most two channel indices. 
Then we have Lemma~\ref{clm:ch2}. 
From the continued fractions in Lemma~\ref{clm:ch2} (0), (1)--(4), 
we obtain Theorem~\ref{thm:channel} (a-1), (b-1)--(b-4) respectively. 
Combining the continued fractions in 
Lemma~\ref{clm:ch2} (5) with those in Lemma~\ref{clm:ch3Exception} (3), 
and the continued fractions in Lemma~\ref{clm:ch2} (7) with those in Lemma~\ref{clm:ch3Exception} (1), 
we have the continued fractions 
\[ [2m + 1, 2n, -2, 2l -1] \text{ with } m \ge 1, \ n \ne 0, \ l \ge 2 \, ,  \]
and 
\[ [2m + 1, 2n, 2, 2l -1] \text{ with } m \ge 1, \ n \ne 0, \ l \le -1 \, .  \]
Combining them, 
we obtain the continued fractions in Theorem~\ref{thm:channel} (c-1). 
Similarly, 
combining Lemma~\ref{clm:ch2} (6) with Lemma~\ref{clm:ch3Exception} (4), and Lemma~\ref{clm:ch2} (8) with Lemma~\ref{clm:ch3Exception} (2), 
we obtain the continued fractions 
\[ [2m + 1, 2n-1, -2, 2l] \text{ with } m \ge 1, \ n \ne 0,1, \ l \ge 1 \, , \]
and 
\[ [2m + 1, 2n-1, 2, 2l] \text{ with } m \ge 1, \ n \ne 0,1, \ l \le -1  \,  \]
respectively. Combining them, 
we obtain Theorem~\ref{thm:channel} (c-2). 
Now we complete the proof of Theorem~\ref{thm:channel} assuming Proposition~\ref{clm:3channels} and Lemmas~\ref{lem:3ch}, \ref{clm:ch3Exception}, and \ref{clm:ch2}. 
\end{proof}

\subsection{Delman's allowable path}\label{subsec:Delman} 



In this subsection, 
we briefly review Delman's branched surfaces and allowable paths 
to introduce Lemma~\ref{lem:3ch}. 

Delman constructed an essential branched surface in a rational tangle space, and studied Dehn surgery on a Montesinos knot in his unpublished preprint~\cite{Delman-unpub}. 
Actually, he gave a construction of such an essential branched surface 
and describe them by using a combinatorial object called an allowable path. 
Based on the work of Li~\cite{Li}, 
Wu~\cite{Wu2012} proposed a sink mark description for branched surfaces. 
This description has made Delman's branched surfaces easy to treat. 
In the following, we briefly review these studies to study Dehn surgeries on two-bridge links. 
Our notations are basically the same as used in \cite{Wu2012}, 
and we assume that the readers are somewhat familiar with those. 
For details about the definitions of terms used in the following, 
refer to \cite{Wu2012} or \cite[Section 5]{Wu1999}. 







As already mentioned in the proof of Theorem~\ref{thm:channel}, 
for a hyperbolic two-bridge link $L_{p/q}$, 
$p/q$ can be expressed by an even continued fraction 
$[b_1, \dots, b_k]$ with $k \ge 3$. 

We can construct the \emph{diagram} $D(p/q)$ associated to $p/q$, 
which is the minimal sub-diagram of 
the Hatcher-Thurston diagram~\cite[Figure 4]{HatcherThurston1985} 
that contains all minimal paths from $1/0$ to $p/q$ 
(see \cite[Section 5]{Wu1999} for example). 
The diagram $D(p/q)$ can be constructed as follows: 
Let $p/q = [b_1 \dots, b_k]$ be an even continued fraction of $p/q$. 
To each $b_i$ is associated a ``fan'' $F_{b_i}$ consisting of 
$|b_i|$ simplices in $D(p/q)$; see Figure~\ref{fig:fan} for the fans $F_4$ and $F_{-4}$. 
The edges labeled $e_1$ are called \emph{initial edges}, 
and the ones labeled $e_2$ are called \emph{terminal edges}. 
The diagram $D(p/q)$ can be constructed by gluing the fans 
$F_{b_1}, \dots, F_{b_k}$ together in such a way that 
the terminal edge of $F_{b_i}$ is glued to the initial edge of $F_{b_{i+1}}$. 
Moreover, if $b_{i} b_{i+1} < 0$, then $F_{b_i}$ and $F_{b_{i+1}}$ have one edge in common, 
and if $b_i b_{i+1} > 0$, then they have a $2$-simplex in common. 
See Figure~\ref{fig:fanEx} for the diagram of $[-2, 2, 4, 2]$. 
As a sub-diagram of the Hatcher-Thurston diagram, 
each vertex of $D(p/q)$ is associated to an irreducible fraction or possibly $1/0$.  
There are three possible parities of the numerators and the denominators of them: 
odd/odd, odd/even, or even/odd, 
denoted by \emph{o/o}, \emph{o/e}, and \emph{e/o}, respectively. 
Note that the three vertices of any simplex in $D(p/q)$ have mutually different parities. 
Also note that each vertex on the initial or terminal edges of $F_{b_i}$ has parity $o/e$ or $e/o$. 
We use the symbol ``$*$'' to indicate vertices with parity $o/o$. 

\begin{figure}[!htb]
\centering
\begin{overpic}[width=.35\textwidth]{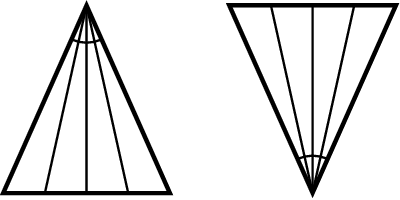}
\put(3,25){$e_1$} 
\put(34,25){$e_2$} 
\put(60,25){$e_1$} 
\put(91,25){$e_2$} 
\put(9.3,-3){$*$} 
\put(30,-3){$*$} 
\put(65.8,50){$*$} 
\put(86.5,50){$*$} 
\end{overpic}
\caption{The fans $F_4$ (left side) and $F_{-4}$ (right side).} 
\label{fig:fan}
\end{figure}

\begin{figure}[!htb]
\centering
\begin{overpic}[width=.3\textwidth]{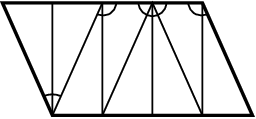}
\put(18.3,46.7){$*$} 
\put(38,-4.5){$*$} 
\put(77,-4.5){$*$} 
\end{overpic}
\caption{The diagram of $[-2, 2, 4, 2]$. 
The fans $F_{-2}$ and $F_2$ have one edge in common, 
and the fans $F_2$ and $F_4$ have a $2$-simplex in common. } 
\label{fig:fanEx}
\end{figure}


Take two simplices in $D(p/q)$ with one edge in common. 
Assume that the two vertices which are not on the common edge are of the parity \textit{o/o}. 
Then each of the arcs indicated in Figure~\ref{fig:channel} is called a \emph{channel} 
which was essentially introduced by Delman \cite{Delman-unpub}.
Note that, 
though a channel connecting two vertices with common parity which is not \textit{o/o} 
can be also defined, 
we only use a channel connecting two vertices with parity \textit{o/o}. 

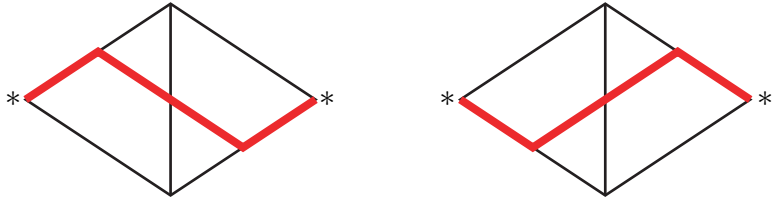
\begin{figure}[!htb]
\centering
\begin{overpic}[width=.6\textwidth]{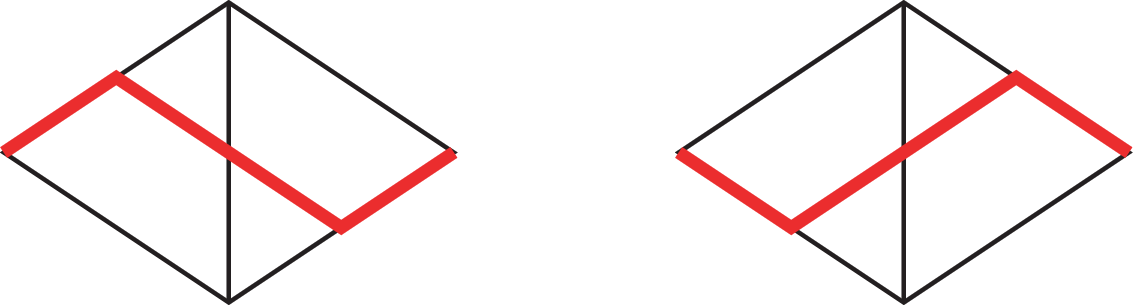}
\put(-2.5,12.5){$*$} 
\put(40.5,12.5){$*$} 
\put(57,12.5){$*$} 
\put(100.5,12.5){$*$} 
\end{overpic}
\caption{Channels.} 
\label{fig:channel}
\end{figure}

A \emph{path} $\gamma$ in $D(p/q)$ is a union of arcs, 
each of which is either an edge of $D(p/q)$ or a channel. 
A path $\gamma$ in $D(p/q)$ is said to be \emph{allowable} if the following three conditions hold (see \cite[Definition 5.2]{Wu1999}). 
\begin{enumerate}[(1)]
\item 
$\gamma$ passes any point of $D(p/q)$ at most once. 
\item 
Other than 
the middle points of channels, $\gamma$ intersects the interior of at
most one edge of any given simplex. 
\item 
$\gamma$ contains at least one channel. 
\end{enumerate}

For brevity, 
by a \emph{path for $p/q$}, 
we mean a path from $1/0$ to $p/q$ in the diagram $D(p/q)$. 
Then we have the following lemma essentially obtained in \cite{Delman-unpub}. 

\begin{lemma}\label{lem:3chBrs} 
If there is an allowable path 
for $p/q$ containing 
three channels, 
then we can construct essential branched surface $\Sigma$ 
in the exterior $E(L_{p/q})$ of $L_{p/q}$. 
Furthermore 
the two components of $S^3 \setminus \mathrm{Int} N(\Sigma)$ 
containing $L_{p/q}$ form a regular neighborhood of $L_{p/q}$, 
$N(L_{p/q}) = V_1 \cup V_2$, and 
each $V_i$ is a cusped solid torus admitting three disjoint meridional cusps on $\partial V_i$ for $i=1,2$.
\end{lemma}
\begin{proof}
The construction of $\Sigma$ was originally introduced by 
Delman~\cite{Delman-unpub}. 
Wu has reformulated the construction of $\Sigma$, 
and reprove that $\Sigma$ is essential; see \cite[Theorem 5.3]{Wu2012}. 
Each channel creates two meridional cusps on $\partial N(L)$ 
as in the proof of \cite[Theorem 5.3]{Wu2012}. 
In addition, one of the two meridional cusps is on $V_1$, 
and the other is on $V_2$ as in the proof of \cite[Lemma 5.3]{Wu1999}. 
One can also show this fact directly 
by drawing the branched surface with a sink mark description 
introduced in \cite{Wu2012}. 
\end{proof}

\begin{remark}
In Delman's branched surface, 
the tangencies at the branch points are introduced by a notion called ``configuration''. 
In this paper, we only use type I configuration (see \cite[Figure 5.2]{Wu1999}). 
\end{remark}




From the above lemma, together with the studies on an essential branched surface in a hyperbolic 3-manifold due to Wu~\cite{Wu1998}, we obtain a proof of Lemma~\ref{lem:3ch}. 

\begin{proof}[Proof of Lemma~\ref{lem:3ch}]
By Lemma~\ref{lem:3chBrs}, we obtain an essential branched surface $\Sigma$ in 
the exterior $E(L_{p/q})$ 
such that the two components of $S^3 \setminus \mathrm{Int} N(\Sigma)$ 
containing $L_{p/q}$ form $N(L_{p/q}) = V_1 \cup V_2$. 
Moreover, each $V_i$ is a cusped solid torus admitting three disjoint 
meridional cusps on $\partial V_i$ for $i=1,2$. 
Then $L_{p/q}$ admits no complete exceptional surgery by the following argument. 
Assume for a contradiction that a Dehn surgery on $L_{p/q}$ along slopes $(\gamma_1, \gamma_2)$ is a complete exceptional surgery. 
Then, neither $\gamma_1$ nor $\gamma_2$ is a meridional slope because each component of $L_{p/q}$ is a trivial knot. Since each $V_i$ has more than one meridional cusps, $\Sigma$ remains an essential branched surface in the $3$-manifold $M$ obtained from $S^3$ after $(\gamma_1, \gamma_2)$-surgery on $L_{p/q}$. Let $F$ be an essential torus in the non-hyperbolic $3$-manifold $M$. 
Then, apply the arguments in the second and third paragraphs of the proof of \cite[Theorem 2.5]{Wu1998} replacing the use of \cite[Theorem 1.6]{Wu1998} by \cite[Theorem 1.9]{Wu1998}. 
We can isotope the essential torus $F$ into $E(L_{p/q})$, a contraction. 
\end{proof}


\subsection{Channel index}\label{subsec:Channelidx}

In this subsection, 
we observe how one can find allowable paths with channels 
in the diagram $D(p/q)$ by using channel indices, 
and prove Proposition~\ref{clm:3channels}. 
Actually, in \cite[Lemma 5.4]{Wu1999}, 
Wu determined rational numbers $p/q$ such that 
$D(p/q)$ does not contain allowable paths with at least two channels. 
Our arguments in this subsection can be regarded as an extension of his arguments. 

We start with introducing a channel index for an even continued fraction. 
As defined in \cite[Section 5]{Wu1999}, 
for an even continued fraction $[b_1, \dots, b_k]$, 
an index $i$ is said to be a \emph{channel index} if 
either $b_i b_{i+1} <0$ or $b_{i} b_{i+1} >4$. 
By this definition, $i$ is not a channel index if and only if 
$b_i b_{i+1} \ge 0$ and $b_{i} b_{i+1} \le 4$, and it is equivalent to 
that $(b_i, b_{i+1}) = (2,2)$ or $(-2,-2)$ since each $b_j$ is even for $j = 1, \dots, k$. 

We then explain how channels in allowable paths can be found if channel indices exist. 
We regard $D(p/q)$ as a graph on a disk $D$, with all vertices on $\partial D$, containing $\partial D$ as a sub-graph. 
Then an edge contained in $\partial D$ is called a \emph{boundary edge}. 
On the other hand, an edge contained in the interior of $D$ is called an \emph{interior edge}. 

First, if $b_i b_{i+1} < 0$, then there is a channel in $F_{b_i} \cup F_{b_{i+1}}$, 
which starts and ends with boundary edges of $D(p/q)$. 
See the left side of Figure~\ref{fig:ChannelOdd} for a channel in $F_2 \cup F_{-2}$. 

Next, if $b_i b_{i+1} > 0$ and $b_i \ge 4$, 
then there is a channel in $F_{b_i} \cup F_{b_{i+1}}$, 
which starts with a boundary edge and ends with an interior edge, 
but its union with a boundary edge of $D(p/q)$ is an allowable path. 
See the center of Figure~\ref{fig:ChannelOdd} for a channel in $F_4 \cup F_2$. 

Similarly, if $b_i b_{i+1} > 0$ and $b_{i+1} \ge 4$, 
then there is a channel which starts with an interior edge and ends with a boundary edge, 
but its union with a boundary edge of $D(p/q)$ is an allowable path. 
See the right side of Figure~\ref{fig:ChannelOdd} for a channel in $F_2 \cup F_4$. 

\begin{figure}[!htb]
\centering
\begin{overpic}[width=.8\textwidth]{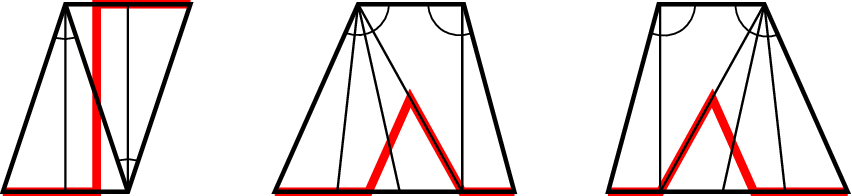}
\put(7,-2){$*$} 
\put(14.3,23.2){$*$} 
\put(38.5,-2){$*$} 
\put(53.5,-2){$*$} 
\put(77,-2){$*$} 
\put(91.5,-2){$*$} 
\end{overpic}
\caption{Paths containing channels in 
$F_2 \cup F_{-2}$, $F_4 \cup F_{2}$, and $F_2 \cup F_4$ respectively. } 
\label{fig:ChannelOdd}
\end{figure}
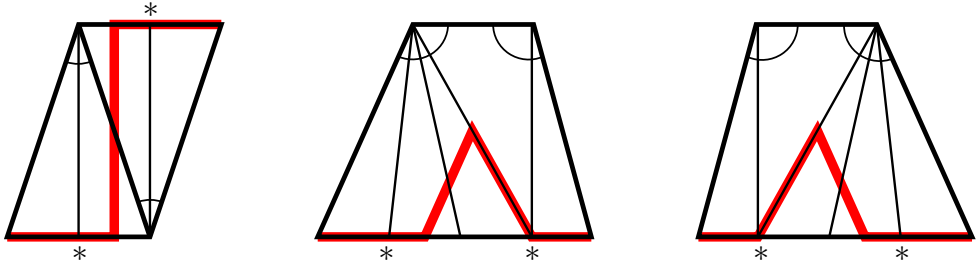


\begin{remark} 
Even if there are $n$ channel indices, 
there do not necessarily exist $n$ channels in a path. 
For example, if $[b_i, b_{i+1}, b_{i+2}] = [2,4,2]$, 
then the indices $i$ and $i+1$ are channel indices. 
However we cannot find a path with two channels in $F_2 \cup F_4 \cup F_2$, 
and can only find a path with one channel; see Figure~\ref{fig:channel242}. 
This situation also arises for $[b_i, b_{i+1}, b_{i+2}] = [-2,-4,-2]$. 
On the other hand, if $[b_i, b_{i+1}, b_{i+2}] = [2, b, 2]$ or $[-2,-b,-2]$ with $b \ge 6$, 
then we have a path with two channels. 
See Figure~\ref{fig:channel262} for $[2,6,2]$ and $[-2,-6,-2]$.  
\end{remark}

\begin{figure}[!htb]
\centering
\begin{overpic}[width=.7\textwidth]{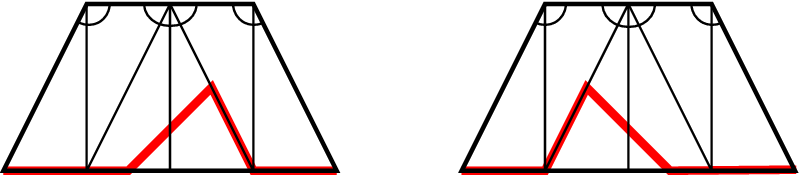}
\put(10,-2){$*$} 
\put(30.7,-2){$*$} 
\put(67,-2){$*$} 
\put(87.8,-2){$*$} 
\end{overpic}
\caption{For $[2,4,2]$, although there are two channel indices, 
we can only find a path with one channel. } 
\label{fig:channel242}
\end{figure}

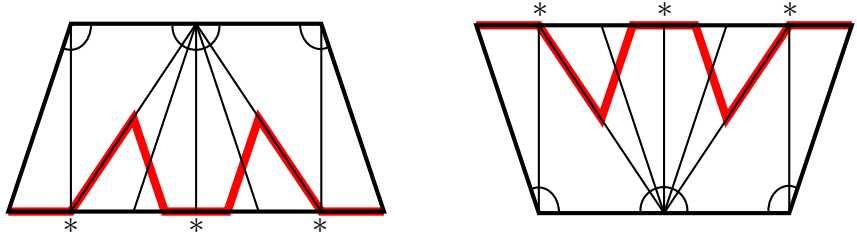
\begin{figure}[!htb]
\centering
\begin{overpic}[width=.7\textwidth]{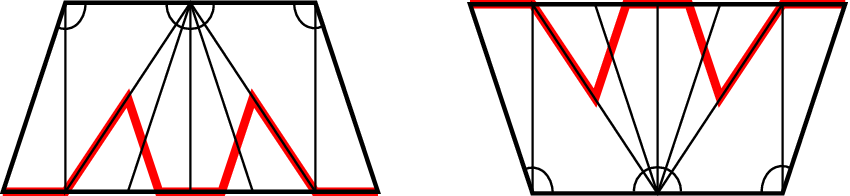}
\put(6.7,-2){$*$} 
\put(21.5,-2){$*$} 
\put(36,-2){$*$} 
\put(61.8,23.5){$*$} 
\put(76.5,23.5){$*$} 
\put(91.2,23.5){$*$} 
\end{overpic}
\caption{For $[2,6,2]$ or $[-2,-6,-2]$, we have a path with two channels. } 
\label{fig:channel262}
\end{figure}

Here we give a proof of Proposition~\ref{clm:3channels}. 

\begin{proof}[Proof of Proposition~\ref{clm:3channels}] 
Let $[b_1,\dots, b_k]$ be an even continued fraction of $p/q$. 
We may assume that $k \ge 3$ as already mentioned.  
As in \cite[Section 5]{Wu1999}, 
we may also assume that either 
\begin{enumerate}[(i)] 
\item 
$b_1 \ge 4$, or 
\item 
$b_1 = 2$ and $b_2 \le -2$. 
\end{enumerate} 
Thus, we can assume that the index $i=1$ is always a channel index. 
%
%

Suppose that there are three channel indices for $p/q=[b_1,\dots, b_k]$. 
Let $i$ and $j$ be the second and the third channel indices respectively ($2 \le i < j < k$). 
We consider a path on $D(p/q)$ starts from $1/0$ on a bottom edge since $b_1$ is positive by the assumption. 

The proof is achieved by case by case argument. 
In each case, we construct an allowable path with three channels 
by combining the channels introduced in Figure~\ref{fig:ChannelOdd}. 

\begin{enumerate}[\text{Case} 1.]
\item
Assume that $b_1 b_2 < 0$, $b_i b_{i+1} <0$, and $b_j b_{j+1} < 0$. 

This is the easiest case. 
Since the path starts with a bottom edge, 
the channel for the index $1$ starts with a bottom edge 
and ends with a top edge of $D(p/q)$. 
Since $1,i,j$ are the first three channel indices and $b_1$ is positive, 
$b_2 , \dots, b_i < 0$ and $b_{i+1}, \dots, b_{j} > 0$, and $b_{j+1}< 0$. 
Thus, the channel for the index $i$ starts with a top edge and ends with a bottom edge, 
and the channel for the index $j$ starts with a top edge and ends with a bottom edge. 
Then the three channels can be connected by boundary edges of $D(p/q)$ 
to become an allowable path 
for $p/q$. 
Note that this works even if $j = i + 1$. 
See Figure~\ref{fig:path1} for typical examples: the paths in the diagrams corresponding to $[2,-2,2,2,-2]$ and $[2,-2,-2,2,-2]$. 

\medskip
\begin{figure}[!htb]
\centering
\begin{overpic}[width=.8\textwidth]{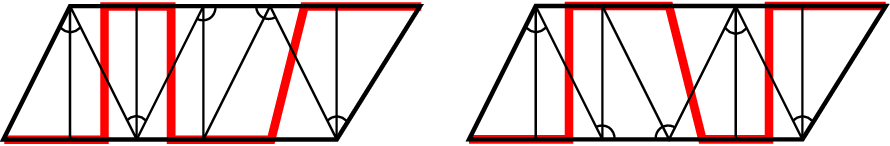}
\put(7.1,-2){$*$} 
\put(22,-2){$*$} 
\put(14.6,16.3){$*$} 
\put(37.6,16.3){$*$} 
\put(59.3,-2){$*$} 
\put(82,-2){$*$} 
\put(67,16.3){$*$} 
\put(89.3,16.3){$*$} 
\end{overpic}
\caption{Allowable paths with three channels 
in the diagrams corresponding to $[2,-2,2,2,-2]$ and $[2,-2,-2,2,-2]$. } 
\label{fig:path1}
\end{figure}
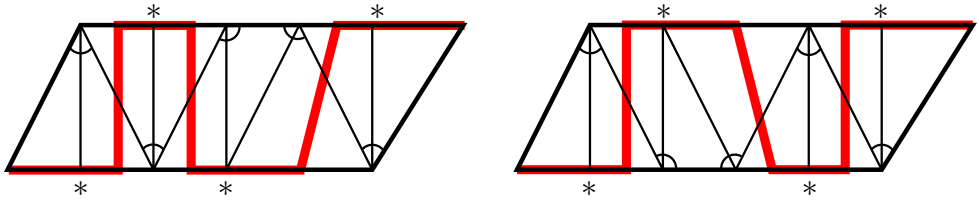

\item 
Assume that $b_1 b_2 < 0$, $b_i b_{i+1} <0$, and $b_j b_{j+1} > 4$. 

By an argument similar to that in Case 1, 
we can find two channels for the indices $1$ and $i$. 
The channel for the index $i$ ends with a bottom edge. 
The path constructed as in the center or the right side of Figure~\ref{fig:ChannelOdd} 
starts and ends with bottom edges. 
So they can be joined with boundary edges of $D(p/q)$ 
to form an allowable path 
for $p/q$. 
Note that this works even if $j = i + 1$. 
See Figure~\ref{fig:path2} for typical examples: the paths in the diagrams corresponding to 
$[2, -2, 2, 2, 4]$ and $[2, -2, 2, 4, 2]$. 

\medskip
\begin{figure}[!htb]
\centering
\begin{overpic}[width=.8\textwidth]{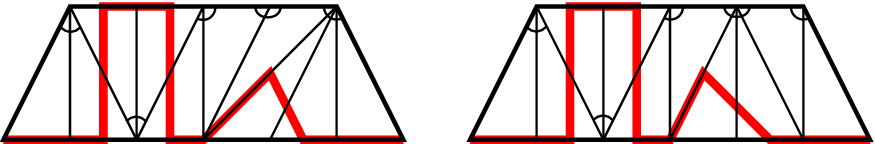}
\put(7.1,-2){$*$} 
\put(22.5,-2){$*$} 
\put(14.8,16.3){$*$} 
\put(37.7,-2){$*$} 
\put(60.5,-2){$*$} 
\put(75.8,-2){$*$} 
\put(68,16.3){$*$} 
\put(91,-2){$*$} 
\end{overpic}
\caption{Allowable paths with three channels 
in the diagrams corresponding to $[2, -2, 2, 2, 4]$ and $[2, -2, 2, 4, 2]$. } 
\label{fig:path2}
\end{figure}
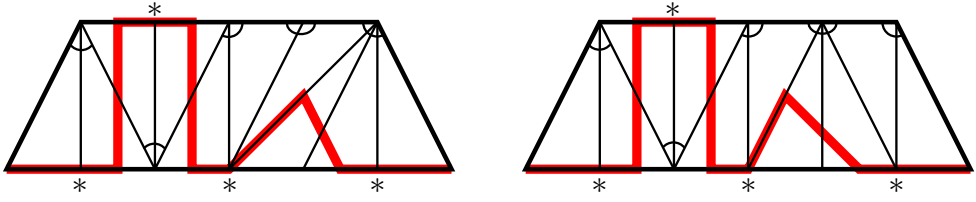

\item
Assume that $b_1 b_2 < 0$, $b_i b_{i+1} > 4$, and $b_j b_{j+1} < 0$. 

In this case, the channel for the index $1$ starts with a bottom edge 
and ends with a top edge, 
the channel for the index $i$ starts and ends with a top edge, 
and the channel for the index $j$ starts with a top edge and ends with a bottom edge. 
Then the three channels can be connected by boundary edges of $D(p/q)$ 
to become an allowable path 
for $p/q$. 
Note that this works even if $j = i + 1$. 
See Figure~\ref{fig:path3} for typical examples: the paths in the diagrams corresponding to 
$[2,-2,-4,2,2]$ and $[2,-2,-2,-4,2]$.

\begin{figure}[!htb]
\centering
\begin{overpic}[width=.8\textwidth]{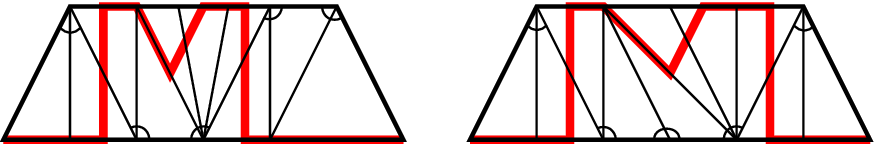}
\put(7.1,-2){$*$} 
\put(30.2,-2){$*$} 
\put(14.8,16.3){$*$} 
\put(25.3,16.3){$*$} 
\put(60.4,-2){$*$} 
\put(91,-2){$*$} 
\put(68,16.3){$*$} 
\put(83.3,16.3){$*$} 
\end{overpic}
\caption{Allowable paths with three channels 
in the diagrams corresponding to $[2,-2,-4,2,2]$ and $[2,-2,-2,-4,2]$. } 
\label{fig:path3}
\end{figure}
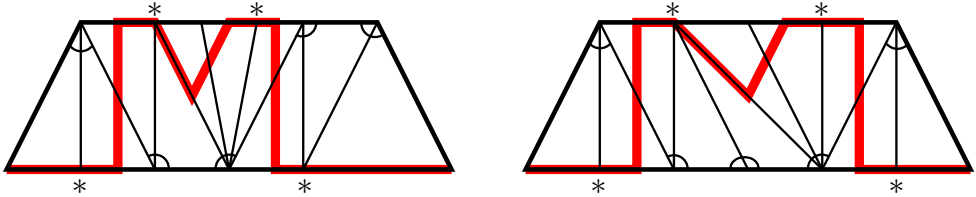

\item
Assume that $b_1 b_2 > 4$, $b_i b_{i+1} < 0$, and $b_j b_{j+1} < 0$. 

We omit details in this case since the proof is similar to that in Case 2. 

\item
Assume that $b_1 b_2 > 4$, $b_i b_{i+1} > 4$, and $b_j b_{j+1} < 0$. 

In this case, we may assume that $b_1 \ge 4$, 
because if $b_1 = 2$, then $b_2 \le -2$ and then $b_1 b_2 <0$. 
Thus, the sub-diagram corresponding to $[2,4,2]$ does not appear even if $i=2$. 
Each of the channels for the indices $1$ and $i$ starts and ends with a bottom edge, 
and the channel for the index $j$ starts with a bottom edge and ends with a top edge. 
Then the three channels can be connected by boundary edges of $D(p/q)$ 
to become an allowable path for $p/q$. 
Note that this works even if $j = i + 1$. 
See Figure~\ref{fig:path5} for typical examples: the paths in the diagrams corresponding to 
$[4,2,4,-2,-2]$ and $[4,4,2,-2,-2]$.

\medskip
\begin{figure}[!htb]
\centering
\begin{overpic}[width=.8\textwidth]{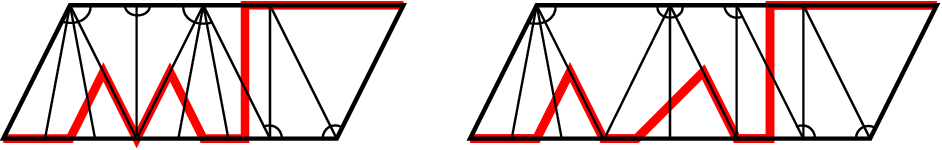}
\put(4,-1.4){$*$} 
\put(13.7,-1.4){$*$} 
\put(27.9,16.3){$*$} 
\put(23.4,-1.4){$*$} 
\put(53.5,-1.4){$*$} 
\put(63.3,-1.4){$*$} 
\put(84.5,16.3){$*$} 
\put(77.4,-1.4){$*$} 
\end{overpic}
\caption{Allowable paths with three channels 
in the diagrams corresponding to $[4,2,4,-2,-2]$ and $[4,4,2,-2,-2]$. } 
\label{fig:path5}
\end{figure}
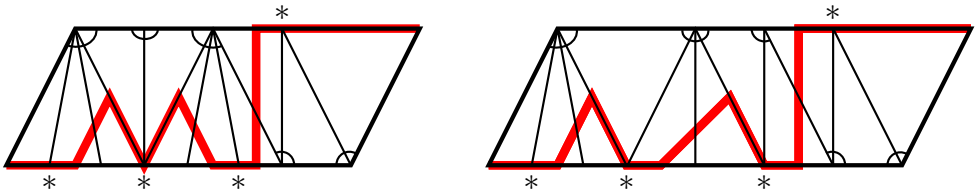

\item
Assume that $b_1 b_2 > 4$, $b_i b_{i+1} < 0$, and $b_j b_{j+1} > 4$. 

In this case, we can prove by the similar argument. 
See Figure~\ref{fig:path6} for typical examples: the paths in the diagrams corresponding to 
$[4,2,-2,-4, -2]$ and $[4,2,2,-2,-4]$. 

\medskip
\begin{figure}[!htb]
\centering
\begin{overpic}[width=.8\textwidth]{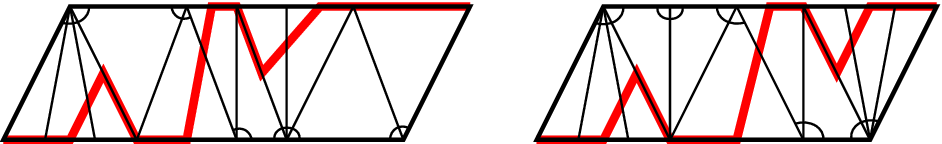}
\put(4,-2){$*$} 
\put(13.7,-2){$*$} 
\put(24.3,15.2){$*$} 
\put(36.7,15.2){$*$} 
\put(60.5,-2){$*$} 
\put(70.3,-2){$*$} 
\put(84.5,15.2){$*$} 
\put(94.2,15.2){$*$} 
\end{overpic}
\caption{Allowable paths with three channels 
in the diagrams corresponding to $[4,2,-2,-4, -2]$ and $[4,2,2,-2,-4]$. } 
\label{fig:path6}
\end{figure}
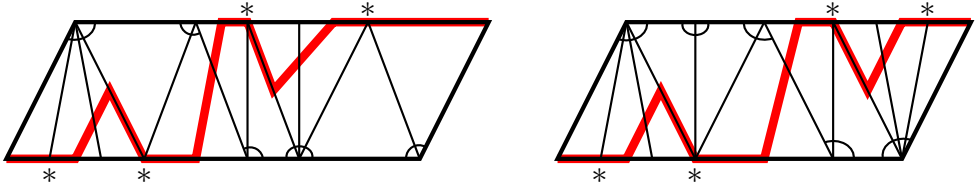

\item
Assume that $b_1 b_2 < 0$, $b_i b_{i+1} > 4$, and $b_j b_{j+1} > 4$. 

In this case, we have to be careful since 
the sub-diagram corresponding to $[-2,-4,-2]$ may appear. 
By the assumptions, $b_2, \dots, b_i, b_{i+1}, \dots, b_j, b_{j+1} \le -2$. 
If $b_2 \le -4$, then $i =2$ and we can find two channels in $F_{b_2} \cup F_{b_3}$ and 
$F_{b_j} \cup F_{b_{j+1}}$ as in Case 5 even if $j = i+1 = 3$. 
See Figure~\ref{fig:path7_1} for typical examples: the paths in the diagrams corresponding to 
$[2,-4,-2,-4,-2]$ and $[2,-4,-4,-2,-2]$. 
Thus, we assume that $b_2 = \dots = b_i = -2$. 
Then $b_{i+1} \le -4$. 
If $b_{i+1} \le -6$, then we can find two channels in $F_{b_i} \cup F_{b_{i+1}}$ and 
$F_{b_j} \cup F_{b_{j+1}}$. 
See Figure~\ref{fig:path7_2} for typical examples: the paths in the diagrams corresponding to 
$[2,-2,-6,-2,-2]$ and $[2,-2,-2,-6,-2]$. 
Thus, we assume that $b_{i+1} = -4$. 
Then $j = i+1$. 
The case where $b_{j+1} = -2$ is excluded in Proposition~\ref{clm:3channels} 
as $[b_1,\dots, b_k] = [b_1, -2, \dots, -2, -4, -2, \dots, -2]$. 
Thus, we may assume that $b_j \le -4$. 
Then we can construct an allowable path for $p/q$ with three channels. 
See Figure~\ref{fig:path7_3} for typical examples: the paths in the diagrams corresponding to 
$[2,-2,-4,-4,-2]$ and $[2,-2,-2,-4,-4]$. 

\medskip 
\begin{figure}[!htb]
\centering
\begin{overpic}[width=.8\textwidth]{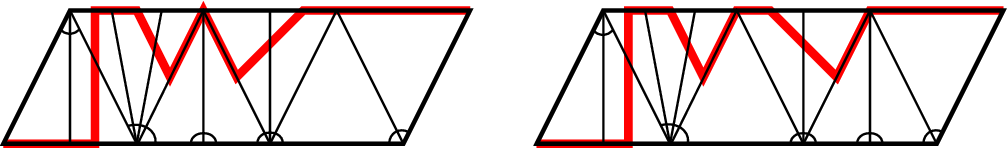}
\put(6.1,-2){$*$} 
\put(10.2,15){$*$} 
\put(19.4,15){$*$} 
\put(32.6,15){$*$} 
\put(59.1,-2){$*$} 
\put(63.2,15){$*$} 
\put(72.1,15){$*$} 
\put(85.3,15){$*$} 
\end{overpic}
\caption{Allowable paths with three channels 
in the diagrams corresponding to $[2,-4,-2,-4,-2]$ and $[2,-4,-4,-2,-2]$. } 
\label{fig:path7_1}
\end{figure}

\medskip 
\begin{figure}[!htb]
\centering
\begin{overpic}[width=.8\textwidth]{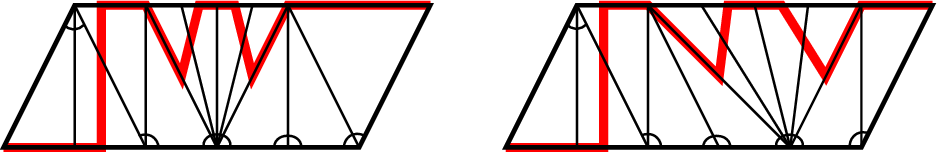}
\put(7.1,-2){$*$} 
\put(14.8,17){$*$} 
\put(22.4,17){$*$} 
\put(30,17){$*$} 
\put(60.7,-2){$*$} 
\put(68.3,17){$*$} 
\put(79.7,17){$*$} 
\put(91,17){$*$} 
\end{overpic}
\caption{Allowable paths with three channels 
in the diagrams corresponding to $[2,-2,-6,-2,-2]$ and $[2,-2,-2,-6,-2]$. } 
\label{fig:path7_2}
\end{figure}

\medskip 
\begin{figure}[!htb]
\centering
\begin{overpic}[width=.8\textwidth]{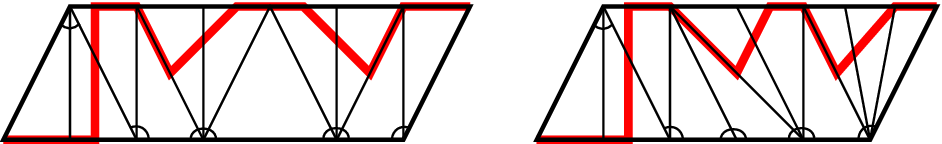}
\put(6.6,-2){$*$} 
\put(13.7,15.7){$*$} 
\put(27.9,15.7){$*$} 
\put(42,15.7){$*$} 
\put(63.3,-2){$*$} 
\put(70.5,15.7){$*$} 
\put(84.5,15.7){$*$} 
\put(94.3,15.7){$*$} 
\end{overpic}
\caption{Allowable paths with three channels 
in the diagrams corresponding to $[2,-2,-4,-4,-2]$ and $[2,-2,-2,-4,-4]$. } 
\label{fig:path7_3}
\end{figure}
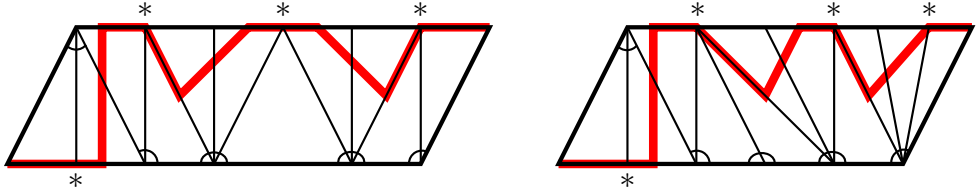

\item
Assume that $b_1 b_2 > 4$, $b_i b_{i+1} > 4$, and $b_j b_{j+1} > 4$. 

In this case, we also have to be careful since 
the sub-diagram corresponding to $[2,4,2]$ may appear. 
We may assume that $b_1 \ge 4$. 
By the assumptions, $b_2, \dots, b_i, b_{i+1}, \dots, b_j, b_{j+1} >0$. 
If $b_2 \ge 4$, then $i =2$ and 
we can find two channels in $F_{b_2} \cup F_{b_3}$ and 
$F_{b_j} \cup F_{b_{j+1}}$ as in Case 5 even if $j = i+1 = 3$. 
See Figure~\ref{fig:path8_1} for typical examples: the paths in the diagrams 
corresponding to $[4,4,4,2,2]$ and $[4,4,2,4,2]$. 
Thus, we assume that $b_2 = \dots = b_i = 2$. 
Then $b_{i+1} \ge 4$. 
If $b_{i+1} \ge 6$, then we can find two channels in $F_{b_i} \cup F_{b_{i+1}}$ and 
$F_{b_j} \cup F_{b_{j+1}}$. 
See Figure~\ref{fig:path8_2} for typical examples: the paths in the diagrams 
corresponding to $[4, 2,6,2,2]$ and $[4,2,2,6,2]$. 
Thus, we assume that $b_{i+1} = 4$. 
Then $j = i+1$. 
The case where $b_j = 2$ is excluded in Proposition~\ref{clm:3channels} 
as $[b_1,\dots, b_k] = [b_1, 2, \dots, 2, 4, 2, \dots, 2]$. 
Thus, we may assume that $b_j \ge 4$. 
Then we can construct an allowable path for $p/q$ with three channels. 
See Figure~\ref{fig:path8_3} for typical examples: the paths in the diagrams 
corresponding to $[4,2,4,4,2]$ and $[4,2,2,4,4]$. 
\end{enumerate} 
Now we complete the proof of Proposition~\ref{clm:3channels}. 
\end{proof}  

\begin{figure}[!htb]
\centering
\begin{overpic}[width=.8\textwidth]{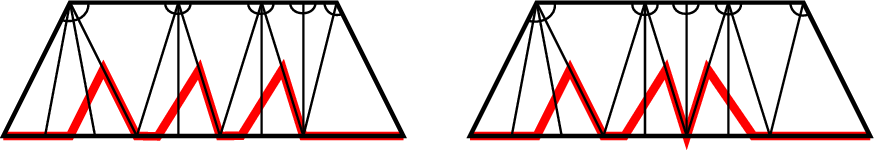}
\put(4.3,-1.5){$*$} 
\put(14.7,-1.5){$*$} 
\put(24.3,-1.5){$*$} 
\put(34,-1.5){$*$} 
\put(57.7,-1.5){$*$} 
\put(68.1,-1.5){$*$} 
\put(77.7,-1.5){$*$} 
\put(87.3,-1.5){$*$} 
\end{overpic}
\caption{Allowable paths with three channels 
in the diagrams corresponding to $[4,4,4,2,2]$ and $[4,4,2,4,2]$. } 
\label{fig:path8_1}
\end{figure}

\medskip 
\begin{figure}[!htb]
\centering
\begin{overpic}[width=.8\textwidth]{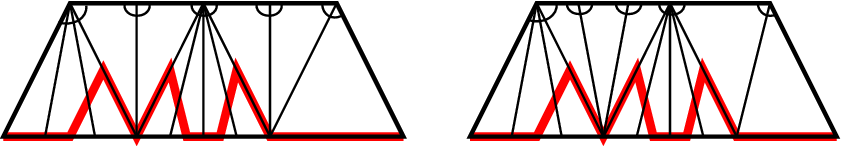}
\put(4.3,-1.7){$*$} 
\put(15.4,-1.7){$*$} 
\put(23.4,-1.7){$*$} 
\put(31.3,-1.7){$*$} 
\put(59.7,-1.7){$*$} 
\put(70.9,-1.7){$*$} 
\put(78.8,-1.7){$*$} 
\put(86.6,-1.7){$*$} 
\end{overpic}
\caption{Allowable paths with three channels 
in the diagrams corresponding to $[4, 2,6,2,2]$ and $[4,2,2,6,2]$. } 
\label{fig:path8_2}
\end{figure}

\medskip
\begin{figure}[!htb]
\centering
\begin{overpic}[width=.8\textwidth]{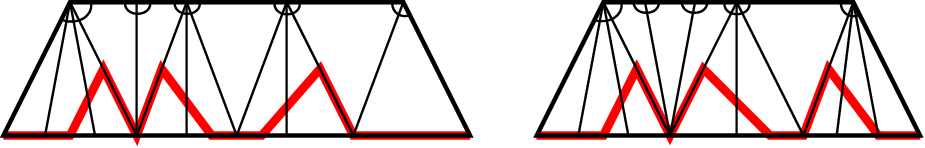}
\put(4.1,-1.7){$*$} 
\put(14.1,-1.7){$*$} 
\put(24.8,-1.7){$*$} 
\put(37.3,-1.7){$*$} 
\put(61.7,-1.7){$*$} 
\put(71.7,-1.7){$*$} 
\put(86.2,-1.7){$*$} 
\put(94.2,-1.7){$*$} 
\end{overpic}
\caption{Allowable paths with three channels 
in the diagrams corresponding to $[4,2,4,4,2]$ and $[4,2,2,4,4]$. } 
\label{fig:path8_3}
\end{figure}
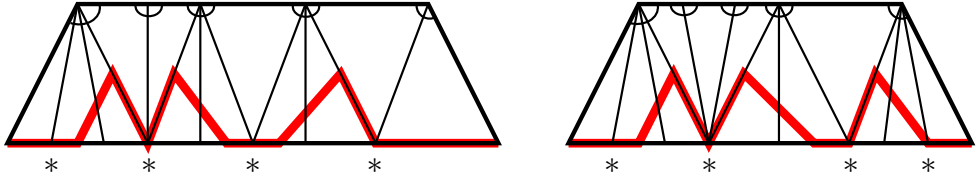

\section{Computer search of exceptional surgeries}\label{sec:computer}
Thanks to Theorem \ref{thm:channel}, it suffices to investigate links in 
Figures \ref{fig:fig_SD-a}, \ref{fig:fig_SD-b}, \ref{fig:fig_SD-c}.
In \cite{MPR}, Martelli-Petronio-Roukema implemented a program which enumerate all candidate exceptional surgeries along a given link.
The code is called \texttt{find\_exceptional\_fillings} (see also \cite{IchiharaMasai}).
It utilizes hyperbolicity verifier HIKMOT \cite{hikmot}, and hence we can verify that all the slopes which do {\em not} appear in the result of
\texttt{find\_exceptional\_fillings} give hyperbolic surgeries.
We modified the code so that it only investigates slopes specified in Figures \ref{fig:fig_SD-a}, \ref{fig:fig_SD-b}, \ref{fig:fig_SD-c}. 
However, if we only use \texttt{find\_exceptional\_fillings} and HIKMOT, 
we get many ``candidate exceptional slopes'' which are quite likely to be hyperbolic.
This is because SnapPy often finds non-geometric solutions (whose solution type is called `contains negatively oriented tetrahedra' in SnapPy)
especially for closed manifolds.
Most of such manifolds have hyperbolic structures, but unfortunately, SnapPy's randomize function does not work in many cases.
To prove those closed manifolds to be hyperbolic, we used \cite[Algorithm 2]{hikmot}.
In the algorithm, by drilling out a closed geodesic and then refilling, we get a new surgery description of a given closed manifold.
By this procedure, we have much better chance to obtain geometric solutions.
In a few cases, \cite[Algorithm 2]{hikmot} does not suffice and we need to take finite coverings and apply \cite[Algorithm 2]{hikmot}.
For more details, see the codes available as ancillary files of arXiv version of this paper.
The results of the calculations are presented in Tables 1--6. 
We remark that the link (c-2) given in Theorem~\ref{thm:channel} yields no elements. 
This completes the proof of the main theorem. 

\begin{figure}[!htb]
\centering
\begin{overpic}[width=.33\textwidth]{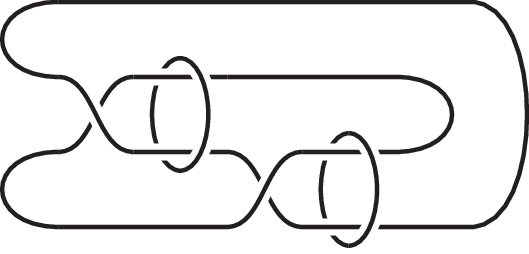}
\put(22,40){$-1/m$} 
\put(54,26){$-1/n$} 
\end{overpic}
\caption{A surgery description of the link (a-1) in Theorem~\ref{thm:channel}.} 
\label{fig:fig_SD-a}
\end{figure}

\begin{figure}[!htb]
\centering
\begin{overpic}[width=.75\textwidth]{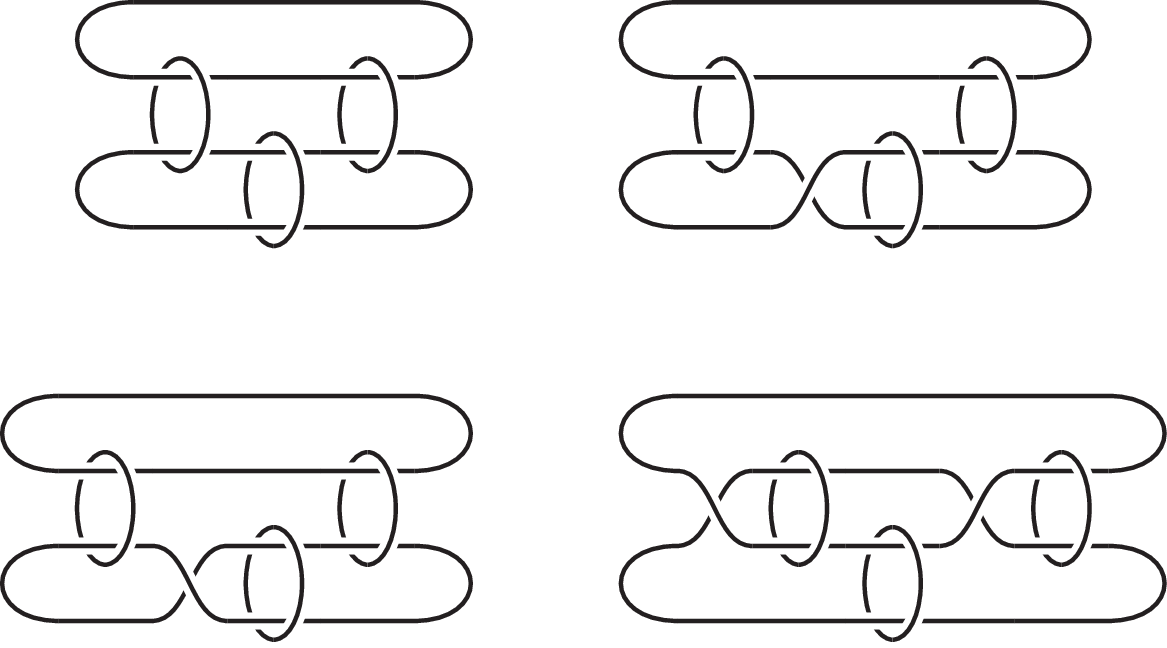}
\put(10,52){$-1/m$} 
\put(19,31){$-1/n$} 
\put(27,52){$-1/l$} 
\put(57,52){$-1/m$} 
\put(71,31){$-1/n$} 
\put(82,52){$1/l$} 
\put(4,18){$-1/m$} 
\put(18,-3){$-1/n$} 
\put(27,18){$-1/l$} 
\put(64,18){$-1/m$} 
\put(71,-3){$-1/n$} 
\put(86,18){$-1/l$} 
\put(0,44.5){(b-1)} 
\put(47,44.5){(b-2)} 
\put(-6,11){(b-3)} 
\put(47,11){(b-4)} 
\end{overpic}
\caption{Surgery descriptions of the links (b-1)--(b-4) in Theorem~\ref{thm:channel}.} 
\label{fig:fig_SD-b}
\end{figure}

\begin{figure}[!htb]
\centering
\begin{overpic}[width=.95\textwidth]{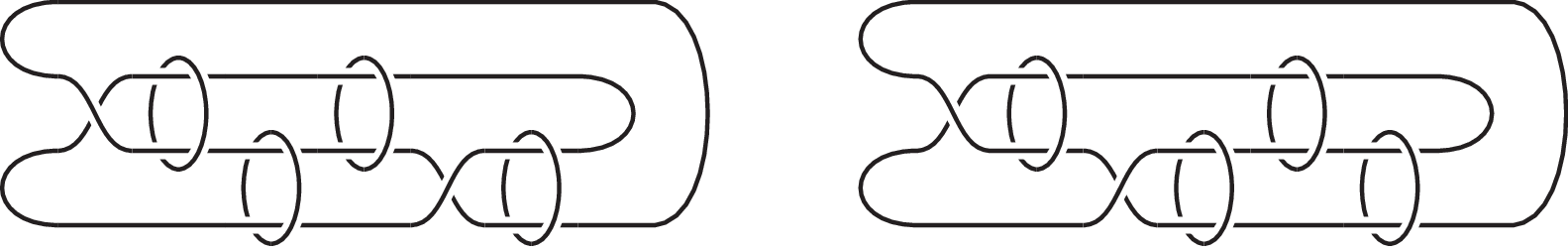}
\put(7.5,13.5){$-1/m$} 
\put(13,-2){$-1/n$} 
\put(30,-2){$-1/l$} 
\put(20,13.5){$\sgn(l)$} 
\put(62,13.5){$-1/m$} 
\put(73,-2){$-1/n$} 
\put(85,-2){$-1/l$} 
\put(80,13.5){$\sgn(l)$} 
\end{overpic}
\caption{Surgery descriptions of the links (c-1) and (c-2) in Theorem~\ref{thm:channel}.} 
\label{fig:fig_SD-c}
\end{figure}
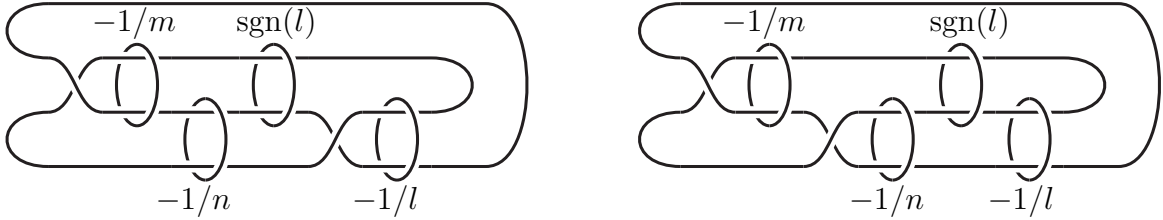

\begin{table}[htb]
\begin{tabular}{|l|l|}
\hline
Link & slopes\\
\hline
$L_{[3,3]}$ &
$(-2,-2)$
$(-2,-1)$
$(-1,-4)$
$(-1,-3)$
$(-1,-1)$
$(5,\frac{4}{3})$
\\\hline
$L_{[3,2 n - 1]}$ &
$(n - 2,n - 2)$
$(n + 3,\frac{2 n - 1}{2})$
\\\hline
$L_{[2 m + 1,-3]}$ &
$(m - 3,\frac{2 m + 1}{2})$
$(m + 2,m + 2)$
\\\hline
$L_{[2 m + 1,3]}$ &
$(m - 1,m - 1)$
$(m + 4,\frac{2 m + 1}{2})$
\\\hline
$L_{[2 m + 1,-5]}$ &
$(m - 5,m)$
\\\hline
$L_{[5,2 n - 1]}$ &
$(n,n + 5)$
\\\hline
$L_{[2 m + 1,5]}$ &
$(m + 1,m + 6)$
\\\hline
$L_{[2 m + 1,2 n - 1]}$ &
$(m + n - 2,m + n + 2)$
$(\frac{2 m + 2 n - 1}{2},\frac{2 m + 2 n + 1}{2})$
\\\hline
\end{tabular}
\caption{Exceptional fillings on Link (a-1)}
\end{table}
\begin{table}[htb]
\begin{tabular}{|l|l|}
\hline
Link & slopes\\
\hline
$L_{[2,2 n,2 l]}$ &
$(l - 1,l - 1)$
\\\hline
\end{tabular}
\caption{Exceptional fillings on Link (b-1)}
\end{table}
\begin{table}[htb]
\begin{tabular}{|l|l|}
\hline
Link & slopes\\
\hline
$L_{[2,2 n - 1,- 2 l]}$ &
$(- l - 1,- l - 1)$
\\\hline
$L_{[2 m,2 n - 1,-2]}$ &
$(m + 1,m + 1)$
\\\hline
\end{tabular}
\caption{Exceptional fillings on Link (b-2)}
\end{table}
\begin{table}[htb]
\begin{tabular}{|l|l|}
\hline
Link & slopes\\
\hline
$L_{[2,2 n + 1,2]}$ &
$(-3,-1)$
$(-2,-2)$
$(-2,-1)$
$(-1,-4)$
$(-1,-1)$
\\\hline
$L_{[2,2 n + 1,2 l]}$ &
$(l - 1,l - 1)$
\\\hline
$L_{[2 m,2 n + 1,2]}$ &
$(m - 1,m - 1)$
\\\hline
\end{tabular}
\caption{Exceptional fillings on Link (b-3)}
\end{table}
\begin{table}[htb]
\begin{tabular}{|l|l|}
\hline
Link & slopes\\
\hline
$L_{[3,2,3]}$ &
$(-3,-1)$
$(-2,-2)$
$(-2,-1)$
$(-1,-4)$
$(-1,-1)$
\\\hline
$L_{[3,2,2 l - 1]}$ &
$(l - 2,l - 2)$
\\\hline
$L_{[2 m + 1,2,3]}$ &
$(m - 1,m - 1)$
\\\hline
$L_{[2 m + 1,2,2 l - 1]}$ &
$(l + m,l + m)$
$(l + m,l + m + 1)$
\\\hline
$L_{[2 m + 1,-2,-3]}$ &
$(m + 2,m + 2)$
\\\hline
$L_{[2 m + 1,-2,2 l - 1]}$ &
$(l + m - 1,l + m)$
$(l + m,l + m)$
\\\hline
$L_{[2 m + 1,2 n,2 l - 1]}$ &
$(l + m + n - 2,l + m + n + 2)$
$(l + m + n - 1,l + m + n + 1)$
$(l + m + n,l + m + n)$
\\\hline
\end{tabular}
\caption{Exceptional fillings on Link (b-4)}
\end{table}
\begin{table}[htb]
\begin{tabular}{|l|l|}
\hline
Link & slopes\\
\hline
$L_{[3,2,2,2 l - 1]}$ &
$(l - 2,l - 2)$
\\\hline
\end{tabular}
\caption{Exceptional fillings on Link (c-1)}
\end{table}

\subsection*{Acknowledgements}
The authors are partially supported by JSPS KAKENHI 
Grant Numbers 18K0327 and 19K03483 and 19K14525, respectively.

\appendix 

\section{Calculations of continued fractions}

We give a lemma used to replace an even continued fraction into a continued fraction not necessarily even. 

\begin{lemma}\label{lem:cf} 
Let $a, k$ be integers with $a \ne 0$, $k \ge 1$, and $y$ a rational number with $y \ne 0$. 
Then we have the following. 
\begin{enumerate}[$(1)$] 
\item 
$[\underbrace{2,\dots,2}_k] = \dfrac{k}{k+1}$. 
\item 
$[\underbrace{-2,\dots,-2}_k] = -\dfrac{k}{k+1}$. 
\item 
$[a, \underbrace{2,\dots,2}_k, y] = [a-1, -(k+1), y-1]$. 
\item 
$[a, \underbrace{-2,\dots,-2}_k, y] = [a+1, k+1, y+1]$. 
\end{enumerate} 
\end{lemma} 
\begin{proof} 
The proof is achieved by an induction on $k$.  
We omit details here. 
\end{proof} 

Here we give a proof of Lemma~\ref{clm:ch3Exception}. 

\begin{proof}[Proof of Lemma~\ref{clm:ch3Exception}] 
First we consider $[a, \underbrace{2,\dots,2}_{b}, 4, \underbrace{2,\dots,2}_{c}]$. 
By Lemma~\ref{lem:cf} (1) and (3), we have 
\begin{align*} 
[a, \underbrace{2,\dots,2}_{b}, 4, \underbrace{2,\dots,2}_{c}]
&= \dfrac{1}{a - \dfrac{1}{2 - \cdots -\dfrac{1}{2 - \dfrac{1}{4 - \dfrac{c}{c+1}}}}} \\
&= [a, \underbrace{2,\dots,2}_{b}, \dfrac{3c+4}{c+1}] \\ 
&= [a-1, -(b+1), \dfrac{3c+4}{c+1} -1 ] \\ 
&= [a-1, -(b+1), 2 + \dfrac{1}{c+1}] \\ 
&= [a-1, -(b+1), 2,  -(c+1)] \, , 
\end{align*} 
where $a$ is even and $b,c$ have the opposite parities, 
and $a \ge 4$, $b \ge 1$, $c \ge 1$. 
Replacing $a -1$ by $2m+1$, 
we have 
\[ [2m+1, -(b+1), 2, -(c+1)] \, \] 
where $b,c$ have the opposite parities, 
and $m \ge 1$, $b \ge 1$, $c \ge 1$. 
If $b$ is odd and $c$ is even, 
then replacing $-(b+1)$ by $2n$ and $-(c+1)$ by $2l -1$, we have 
\[[2m+1, 2n, 2, 2l -1] \text{ \ with \ } m \ge 1, \ n \le -1,  \ l \le -1. \] 
Then we obtain the continued fractions in Lemma~\ref{clm:ch3Exception} (1). 
If $b$ is even and  $c$ is odd, 
then replacing $-(b+1)$ by $2n-1$ and $-(c+1)$ by $2l$, we have 
\[[2m+1, 2n-1, 2, 2l] \text { \ with \ } m \ge 1, \ n \le -1, \ l \le -1. \] 
Then we obtain the continued fractions in Lemma~\ref{clm:ch3Exception} (2). 

Next we consider $[a', \underbrace{-2,\dots, -2}_{b}, -4, \underbrace{-2,\dots,-2}_{c}]$. 
By Lemma~\ref{lem:cf} (2) and (4), we have 
\begin{align*} 
[a', \underbrace{-2,\dots,-2}_{b}, -4, \underbrace{-2,\dots,-2}_{c}]
&= \dfrac{1}{a' - \dfrac{1}{-2 - \cdots -\dfrac{1}{-2 - \dfrac{1}{-4 - \left(-\dfrac{c}{c+1}\right)}}}} \\
&= [a', \underbrace{-2,\dots,-2}_{b}, \dfrac{-3c-4}{c+1}] \\ 
&= [a' +1, b+1, \dfrac{-3c-4}{c+1} +1 ] \\ 
&= [a' +1, b+1, -2 - \dfrac{1}{c+1}] \\ 
&= [a'+1, b+1, -2,  c+1] \, , 
\end{align*} 
where $a'$ is even and $b,c$ have the opposite parities, 
and $a' \ge 2$, $b \ge 1$, $c \ge 1$. 
Replacing $a' +1$ by $2m+1$, 
we have 
\[ [2m+1, b+1, -2, c+1] \, \] 
where $b,c$ have the opposite parities, 
and $m \ge 1$, $b \ge 1$, $c \ge 1$. 
If $b$ is odd and $c$ is even, 
then replacing $b+1$ by $2n$ and $c+1$ by $2l -1$, we have 
\[[2m+1, 2n, -2, 2l -1] \text { \ with \ } m \ge 1, \ n \ge 1, \ l \ge 2. \] 
Then we obtain the continued fractions in Lemma~\ref{clm:ch3Exception} (3). 
If $b$ is even and $c$ is odd, 
then replacing $b+1$ by $2n-1$ and $c+1$ by $2l$, we have 
\[[2m+1, 2n-1, -2, 2l] \text { \ with \ } m \ge 1, \ n \ge 2, \ l \ge 1. \] 
Then we obtain the continued fractions in Lemma~\ref{clm:ch3Exception} (4), 
and complete the proof of Lemma~\ref{clm:ch3Exception}. 
\end{proof}


Next we give a proof of Lemma~\ref{clm:ch2}. 

\begin{proof}[Proof of Lemma~\ref{clm:ch2}] 
Let $[b_1, \dots, b_k]$ be an even continued fraction of $p/q$ 
with at most two channel indices and $k \ge 3$. 
As mentioned in the proof of Proposition~\ref{clm:3channels},  
we may assume that either 
(i) $b_1 \ge 4$, or 
(ii) $b_1 = 2$ and $b_2 \le -2$. 

First we consider an even continued fraction with just one channel index. 
Since the index $i=1$ is a channel index, an even continued fraction 
$[b_1, \dots , b_k]$ 
with just one channel index is one of either 
\[ [b_1, \underbrace{2,\dots,2}_{k-1}] \, , \,  
[b_1, \underbrace{-2,\dots,-2}_{k-1}] \, , 
\text{ or }[b_1, b_2] \]
with 
$|b_2|\ge 4$. 
Since $k \ge 3$, $[b_1, b_2]$ is unsuitable. 
We consider the two cases where (i) $b_1 \ge 4$, and (ii) $b_1 = 2$ and $b_2 \le -2$. 
\begin{enumerate}[(i)] 
\item 
Assume that $b_1 \ge 4$. 
By Lemma~\ref{lem:cf} (1), we have 
\[ [b_1, \underbrace{2,\dots,2}_{k-1}] = \dfrac{1}{b_1 - \dfrac{k-1}{k}} 
= \dfrac{1}{(b_1 -1) - \dfrac{1}{-k}}
= [b_1 - 1, -k] \, . \] 
Replacing $b_1 - 1$ by $2m +1$, and $-k$ by $2n-1$, we have 
\[ [b_1, \underbrace{2,\dots,2}_{k-1}] = [2m+1, 2n-1]\, , \]  
where $m,n$ are integers with $m \ge 1$, $n \le -1$. 
Similarly, 
by using Lemma~\ref{lem:cf} (2) and 
replacing $b_1 + 1$ by $2m+1$ and $k$ by $2n-1$, 
we have 
\[ [b_1, \underbrace{-2,\dots,-2}_{k-1}] 
= \dfrac{1}{b_1 - \left(- \dfrac{k-1}{k}\right)} 
= [b_1 + 1, k] = [2m+1, 2n-1]\, , \]  
where $m,n$ are integers with $m \ge 2$, $n \ge 2$. 

\item 
Assume that $b_1 = 2$. 
Using Lemma~\ref{lem:cf} (2) and replacing $k$ by $2n-1$, we have 
\[ [2, \underbrace{-2,\dots,-2}_{k-1}] 
= \dfrac{1}{2 - \left(- \dfrac{k-1}{k}\right)} 
= [3,k] = [3, 2n-1]\, , \]  
where $n$ is an integer with $n \ge 2$. 
\end{enumerate} 
Combining (i) with (ii), we obtain the continued fractions in Lemma~\ref{clm:ch2} (0). 

\smallskip

Next we consider an even continued fraction with just two channel indices. 
As in the former case, 
we consider the two cases where (i) $b_1 \ge 4$, and (ii) $b_1 = 2$ and $b_2 \le -2$. 
Recall that the index $i=1$ is a channel index in each case. 
\begin{enumerate}[(i)] 
\item 
Assume that $b_1 \ge 4$. 
Let $a = b_1$. 
In this case, by the definition of a channel index, 
an even continued fraction with just two channel indices is one of the following: 
\begin{enumerate}[(1)]
\item[(i-1-1)] 
$[a,b,c]$, 
where $b,c$ are even, and $|b| \ge 4$, $|c| \ge 4$. 

\item[(i-1-2)] 
$[a, b, \underbrace{2, \dots, 2}_c ]$, 
where $b$ is even and $c$ is odd, 
and $|b| \ge 4$, $c \ge 1$. 

\item[(i-1-3)] 
$[a,b, \underbrace{-2, \dots, -2}_c ]$, 
where $b$ is even and $c$ is odd, 
and $|b| \ge 4$, $c \ge 1$. 
 
\item[(i-2-1)] 
$[a, \underbrace{2, \dots, 2}_b, c]$, 
where $b$ is odd and $c$ is even, 
and $b \ge 1$, $|c| \ge 4$. 

\item[(i-2-2)] 
$[a, \underbrace{2, \dots, 2}_b, \underbrace{-2, \dots, -2}_c]$, 
where $b,c$ have the same parity, 
and $b \ge 1$, $c \ge 1$. 

\item[(i-3-1)] 
$[a, \underbrace{-2, \dots, -2}_b, c]$, 
where $b$ is odd and $c$ is even, 
and $b \ge 1$, $|c| \ge 4$. 

\item[(i-3-2)] 
$[a, \underbrace{-2, \dots, -2}_b, \underbrace{2, \dots, 2}_c]$, 
where $b,c$ have the same parity, 
and $b \ge 1$, $c \ge 1$. 
\end{enumerate} 

\item 
Assume that $b_1 = 2$ and $b_2 \le -2$. 
By the same enumeration, 
we have the following: 
\begin{enumerate}[(1)] 
\item[(ii-1-1)] 
$[2,b,c]$, where $b,c$ are even, and $b \le -4$, $|c| \ge 4$. 
\item[(ii-1-2)] 
$[2, b, \underbrace{2, \dots, 2}_c ]$, where $b$ is even and $c$ is odd, 
and $b \le -4$, $c \ge 1$. 
\item[(ii-1-3)] 
$[2, b, \underbrace{-2, \dots, -2}_c]$, where $b$ is even and $c$ is odd, 
and $b \le -4$, $c \ge 1$. 
\item[(ii-3-1)] 
$[2, \underbrace{-2, \dots, -2}_b, c]$, where $b$ is odd and $c$ is even, 
and $b \ge 1$, $|c| \ge 4$. 
\item[(ii-3-2)] 
$[2, \underbrace{-2, \dots, -2}_b, \underbrace{2, \dots, 2}_c]$, 
where $b,c$ have the same parity, 
and $b \ge 1$, $c \ge 1$. 
\end{enumerate} 
\end{enumerate} 
Combining (i-1-1) with (ii-1-1), 
(i-1-2) with (ii-1-2), (i-1-3) with (ii-1-3), 
we obtain the following respectively. 
\begin{enumerate}[(1)] 
\item[{\rm (1-1)}] 
$[a,b,c]$, where $a,b,c$ are even, and $a \ge 2$, $|b| \ge 4$, $|c| \ge 4$. 
\item[{\rm (1-2)}] 
$[a, b, \underbrace{2, \dots, 2}_c ]$, where $a,b$ are even and $c$ is odd, 
and $a \ge 2$, $|b| \ge 4$, $c \ge 1$. 
\item[{\rm (1-3)}] 
$[a,b, \underbrace{-2, \dots, -2}_c ]$, where $a,b$ are even and $c$ is odd, 
and $a \ge 2$, $|b| \ge 4$, $c \ge 1$. 
\end{enumerate} 
Note that, in (1-1)--(1-3), if $a = 2$, then $b \le -4$ holds. 

From (i-2-1) and (i-2-2), we obtain the following respectively. 
\begin{enumerate}[(1)] 
\item[{\rm (2-1)}] 
$[a, \underbrace{2, \dots, 2}_b, c]$, where $a, c$ are even and $b$ is odd, 
and $a \ge 4$, $b \ge 1$, $|c| \ge 4$. 
\item[{\rm (2-2)}] 
$[a, \underbrace{2, \dots, 2}_b, \underbrace{-2, \dots, -2}_c]$, where $a$ is even and $b,c$ have the same parity, 
and $a \ge 4$, $b \ge 1$, $c \ge 1$. 
\end{enumerate} 

Combining (i-3-1) with (ii-3-1), (i-3-2) with (ii-3-2), 
we obtain the following respectively. 
\begin{enumerate}[(1)] 
\item[{\rm (3-1)}] 
$[a, \underbrace{-2, \dots, -2}_b, c]$, where $a, c$ are even and $b$ is odd, 
and $a \ge 2$, $b \ge 1$, $|c| \ge 4$. 
\item[{\rm (3-2)}] 
$[a, \underbrace{-2, \dots, -2}_b, \underbrace{2, \dots, 2}_c]$, where $a$ is even and $b,c$ have the same parity, 
and $a \ge 2$, $b \ge 1$, $c \ge 1$. 
\end{enumerate} 

The continued fractions (1-1) coincide with 
those of in Lemma~\ref{clm:ch2} (1) 
by replacing $a$ by $2m$, $b$ by $2n$, $c$ by $2l$. 

For the continued fraction (1-2), 
using Lemma~\ref{lem:cf} (1), we have 
\[ [a, b, \underbrace{2, \dots, 2}_c ] 
= \dfrac{1}{a - \dfrac{1}{b - \dfrac{c}{c+1}}} 
= \dfrac{1}{a - \dfrac{1}{(b - 1) - \dfrac{1}{-(c+1)}}}  
= [a, b-1, -(c+1)] \, , \] 
where $a,b$ are even and $c$ is odd, 
and $a \ge 2$, $|b| \ge 4$, $c \ge 1$. 
In addition, if $a = 2$, then $b \le -4$ holds. 
Replacing $a$ by $2m$, $b -1$ by $2n-1$, $-(c+1)$ by $-2l$, we have 
\[[2m, 2n-1, -2l] \] 
with $m \ge 1$, $|n| \ge 2$, $l \ge 1$. 
In addition, if $m = 1$, then $n \le -2$ holds. 
Then we obtain the continued fractions listed in Lemma~\ref{clm:ch2} (2). 

By the similar argument, 
for the continued fraction (1-3), 
using Lemma~\ref{lem:cf} (2) and replacing $a,b,c$ suitably, 
we obtain the continued fractions listed in Lemma~\ref{clm:ch2} (3). 


For the continued fraction (2-1), 
using Lemma~\ref{lem:cf} (3), we have 
\[ [a, \underbrace{2, \dots, 2}_b, c] = [a-1, -(b+1), c-1] \, , \] 
where $a, c$ are even and $b$ is odd, and $a \ge 4$, $b \ge 1$, $|c| \ge 4$. 
Replacing $a -1$ by $2m+1$, $-(b+1)$ by $2n$, $c-1$ by $2l-1$, we have 
\begin{equation} 
[2m+1, 2n, 2l-1] \text{ with } m \ge 1, \ n \le -1, \ |l| \ge 2.  \tag{$*$}
\end{equation} 

For the continued fraction (2-2), 
using Lemma~\ref{lem:cf} (2), we have 
\[ [a, \underbrace{2, \dots, 2}_b, \underbrace{-2, \dots, -2}_c] 
= \dfrac{1}{a - \dfrac{1}{2 - \cdots -\dfrac{1}{2 - \left(-\dfrac{c}{c+1}\right)}}} 
= [a, \underbrace{2, \dots, 2}_b, -\dfrac{c+1}{c}] \, . \] 
Then, by Lemma~\ref{lem:cf} (3), we have 
\[ 
[a, \underbrace{2, \dots, 2}_b, -\dfrac{c+1}{c}]
= [a-1, -(b+1), -\dfrac{c+1}{c} -1] 
= [a-1, -(b+1), -2, c] \, , 
\] 
where $a$ is even and $b,c$ have the same parity, 
and $a \ge 4$, $b \ge 1$, $c \ge 1$. 
Replacing $a -1$ by $2m+1$, 
we have 
\[ [2m+1, -(b+1), -2, c] \, \] 
where $b,c$ have the same parities, 
and $m \ge 1$, $b \ge 1$, $c \ge 1$. 
In this case, if $c = 1$, then 
$b$ is odd, and we have 
\[ 
[2m+1, -(b+1), -2, 1] = [2m+1, -(b+1), -3] \] 
which can be regard as the case where $l = -1$ in ($*$) by replacing $-(b+1)$ by $2n$. 

Thus, from (2-1) and (2-2), 
we obtain the following three families of continued fractions: 
\begin{enumerate}[(a)]
\item 
$[2m+1, 2n, 2l-1]$ with $m \ge 1$, $n \le -1$, $l \ne 0, 1$.  
\item 
$[2m+1, 2n, -2, 2l-1]$ with $m \ge 1$, $n \le -1$, $l \ge 2$. 
\item 
$[2m+1, 2n-1, -2, 2l]$ with $m \ge 1$, $n \le -1$, $l \ge 1$. 
\end{enumerate}


By the similar argument for the continued fractions (3-1) and (3-2), 
we have the following two families: 
\begin{enumerate}[(a)'] 
\item 
$[2m+1,2n, 2l -1]$ with $m \ge 1$, $n \ge 1$, $l \ne 0,1$.  
\item 
$[2m+1, 2n, 2, 2l-1]$ with $m\ge 1$, $n \ge 1$, $l \le -1$. 
\item 
$[2m+1, 2n-1, 2, 2l]$ with $m\ge 1$, $n \ge 2$, $l \le -1$. 
\end{enumerate}

Combining (a) with (a)', we obtain the continued fractions listed in Lemma~\ref{clm:ch2} (4). 
The continued fractions (b), (c), (b)', (c)' coincide with those listed in 
Lemma~\ref{clm:ch2} (5), (6), (7), (8) respectively. 
Now we complete the proof of Lemma~\ref{clm:ch2}. 
\end{proof}

\end{document}